\documentclass[12pt]{amsart}
\usepackage{amsmath,amsthm,amssymb,amsfonts}
\usepackage{latexsym}
\usepackage{enumerate}
\usepackage{verbatim}
\usepackage{multicol}
\usepackage{graphicx}
\usepackage{epsfig}
\usepackage{color}
\usepackage{bezier}
\usepackage{curves}
\usepackage{fullpage}

\newtheorem{theorem}{Theorem}

\newtheorem{remark}{Remark}
\newtheorem{lemma}[theorem]{Lemma}

\newtheorem{proposition}[theorem]{Proposition}
\newtheorem{corollary}[theorem]{Corollary}

\newtheorem{df}{Definition}

\newtheorem{conjecture}{Conjecture}

\newenvironment{myremark}[1][]{%
  \begin{remark}[#1]$ $\par\nobreak\ignorespaces
}{%
  \end{remark}
}

\DeclareMathOperator{\comp}{{Comp}}

\DeclareMathOperator{\incg}{Inc}
\DeclareMathOperator{\inc}{inc}

\newcommand{\NN}{{\mathbb N}}
\newcommand{\ZZ}{{\mathbb Z}}

\begin{document}

\title{The Aharoni--Korman conjecture for $N$-free posets with no infinite antichain}


\author[I. Zaguia]{Imed Zaguia*}\thanks{*The author is supported by the Canadian Defense Academy Research Program (CDARP)}
\address{Department of Mathematics \& Computer Science,  Royal Military College,
P.O.Box 17000, Station Forces, Kingston, Ontario, Canada K7K 7B4}
\email{zaguia@rmc.ca}

\date{\today}


\begin{abstract}We give a necessary and sufficient condition for a $P_4$-free graph to be a cograph. This allows us to obtain a simple proof of the fact that finite $P_4$-free graphs are finite cographs. We also prove that $N$-free  chain complete posets and $N$-free posets with no infinite antichains are series-parallel.

As a consequence, we obtain that every $N$-free poset with no infinite antichain has a chain and a partition into antichains so that each part intersects the chain. This answers a conjecture of Aharoni and Korman (Order \textbf{9} (1992) 245--253) in this case.
\end{abstract}

\maketitle

\section{Introduction}

The graphs we consider are undirected, simple and have no loops. That is, a {\it graph} is a
pair $G:=(X, E)$, where $E$ is a subset of $[X]^2$, the set of $2$-element subsets of $X$. Elements of $V$ are the {\it vertices} of $G$ and elements of $ E$ its {\it edges}. The {\it complement} of $G$ is the graph ${\overline G}$ whose vertex set is $X$ and edge set ${\overline { E}}:=[X]^2\setminus  E$. If $A$ is a subset of $X$, the pair $G_{\restriction A}:=(A,  E\cap [A]^2)$ is the \emph{graph induced by $G$ on $A$}. Unless otherwise stated, the graphs we consider are not necessarily finite.

\begin{df}
A graph $G$ is called \emph{complement reducible}, or a cograph, if every induced
subgraph of $G$ with at least two vertices is either disconnected or the complement to a
disconnected graph.
\end{df}

Cographs were introduced in the early 1970s by Lerchs \cite{Lerchs71,Lerchs72} who
studied their structural and algorithmic aspects. Cographs independently arose in the study of empirical logic where they were called hereditary dacey graphs (see \cite{foulis,sumner74}).

\begin{df}  The path $P_n$ is the graph whose vertex set is $\{1,\cdots,n\}$, for a positive integer $n$, and edge set $E=\{\{i,j\} : |i-j|=1\}$. A graph is $P_n$-free if it has no $P_n$ as an induced subgraph.
\end{df}

\begin{myremark}
  \begin{itemize}
  \item A graph $G$ is $P_4$-free if and only if its complement $\overline{G}$ is $P_4$-free.
  \item The class of $P_4$-free graphs is hereditary, that is, if $G$ is a $P_4$-free and $H$ is an induced subgraph of $G$, then $H$ is a $P_4$-free.
\end{itemize}
\end{myremark}

The following characterization of finite cographs was obtained by several authors and can be found in \cite{Lerchs71,foulis,sumner74}. For several equivalent characterization of finite cographs see Theorems 11.3.3 and 11.3.5 from \cite{bvs}.

\begin{theorem}\label{thm:1}
 A \emph{finite} graph is a cograph if and only if it has no $P_4$ as induced subgraph.
\end{theorem}

The implication $\Rightarrow$ is true even if the graph is assumed to be infinite. On the other hand the implication $\Leftarrow$ is false in the infinite case. Consider the graph $\mathcal{H}=(\ZZ,E)$ whose vertex set is the set of integers $\ZZ$ and edge set $E=\{\{i,j\} : i \mbox{ is even and } i<j\}$. It follows from the definition of $\mathcal{H}$ that the even integers form an infinite clique and the odd integers form an infinite independent set. We now prove that $\mathcal{H}$ does not have $P_4$ as an induced subgraph, yet $\mathcal{H}$ and $\overline{\mathcal{H}}$ are connected. Suppose for a contradiction that there are integers $i,j,k,l$ such that $\mathcal{H}_{\restriction \{i,j,k,l\}}$ is $P_4$ and assume without loss of generality that $\{i,j\},\{j,k\},\{k,l\}$ are its edges. Then $i$ cannot be even because otherwise $k$ and $l$ must be odd and hence cannot form an edge contradicting our assumption. By symmetry $l$ cannot be even. It follows then that $i$ and $l$ are odd and $j$ and $k$ are even. Since $\{i,j\},\{k,l\}$ are edges we infer that $j<i$ and $k<l$. Since $\{j,l\}$ is not an edge and $j$ is even and $l$ is odd it follows that $l<j$. Hence we have $k<l<j<i$ and therefore $\{k,i\}$ is an edge contradicting our assumption. This proves that $\mathcal{H}$ does not have $P_4$ as an induced subgraph. One can easily see that $\mathcal{H}$ and its complement are connected hence $\mathcal{H}$ is not a cograph.\\

In this paper we generalize Theorem \ref{thm:1} to infinite graphs and obtain in route a simple proof of it. Let $G=(X,E)$ be a graph. For $x\in X$ we denote by $N(x):=\{y\in X : \{x,y\}\in E\}$ and set $\inc(x):=X\setminus (N(x)\cup \{x\})$. Define
\[N_x:=\{y\in N(x) : y \mbox{ is adjacent to all elements of } \inc(x)\}.\]

\begin{theorem}\label{thm:2}Let $G=(X,E)$ be a connected $P_4$-free graph of any cardinality. There exists $x\in X$ such that $N_x\neq \emptyset$ if and only if ${\overline G}$ is not connected.
\end{theorem}

As a consequence we obtain the following result which clearly implies Theorem \ref{thm:1}.

\begin{corollary}\label{cor:1}Let $G=(X,E)$ be a connected $P_4$-free graph of any cardinality. If there exists $x\in X$ such that $G_{\restriction \inc(x)}$ has finitely many connected components, then ${\overline G}$ is not connected.
\end{corollary}

The proof of these two results will be given in Section \ref{general}.\\

There is an interplay between $P_4$-free graphs and a class of partially ordered sets. Throughout, $P :=(X, \leq)$ denotes a partially ordered set, poset for short. For $x,y\in V$ we say that $x$ and $y$ are \emph{comparable} if $x\leq y$ or $y\leq x$; otherwise we say that $x$ and $y$
are \emph{incomparable}. A set of pairwise incomparable elements is called an \emph{antichain}.
A \emph{chain} is a totally ordered set. The \emph{comparability graph}, respectively the \emph{incomparability graph}, of $P$ is the graph, denoted by $\comp(P)$, respectively $\incg(P)$, with vertex set $X$ and edges the pairs $\{u,v\}$ of comparable distinct vertices (that is, either $u< v$ or $v<u$) respectively incomparable vertices.

A 4-tuple $(a, b, c, d)$ of distinct elements of $X$ is an $N$ in $P$ if $b>a$ and $b>c$, $d>c$ and if these are the only strict comparabilities between the elements $a,b,c,d$. The poset $P$ is $N$-{\it free} if it does  not contain an $N$. We mention the following result that connects $P_4$-free graphs to $N$-free posets.

\begin{proposition}\label{prop:1} A graph is $P_4$-free if and only if it is the comparability graph of an $N$-{\it free} poset.
\end{proposition}

Indeed, using the characterisation of comparability graphs \cite{houri,gilmore} one can easily see that $P_4$-free graphs are comparability graphs. Jung \cite{jung} gave a direct proof of Proposition \ref{prop:1} without using the characterisation of comparability graphs from \cite{houri,gilmore}.

Finite $N$-free posets can be constructed from the 1-element poset using disjoint and linear sum. Let $I$ be a poset such that $|I|\geq 2$ and let $\{P_{i}\}_{i\in I}$ be a family of pairwise disjoint nonempty posets that are all disjoint from $I$. The \emph{lexicographical sum} $\displaystyle \sum_{i\in I} P_{i}$ is the poset defined on $\displaystyle \bigcup_{i\in I} P_{i}$ by $x\leq y$ if and only if
\begin{enumerate}[(a)]
\item There exists $i\in I$ such that $x,y\in P_{i}$ and $x\leq y$ in $P_{i}$; or
\item There are distinct elements $i,j\in I$ such that $i<j$ in $I$ and $x\in P_{i}$ and $y\in P_{j}$.
\end{enumerate}


If $I$ is a totally ordered set, then $\displaystyle \sum_{i\in I}
P_{i}$ is called a \emph{linear sum}. On the other hand, if $I$ is unordered (antichain), then $\displaystyle \sum_{i\in I} P_{i}$ is called a \emph{disjoint sum}. A poset is \emph{connected} if it is not a disjoint sum and it is \emph{co-connected} if it is not a linear sum.

Let $\mathcal{SP}$ denote the set of finite posets containing the 1-element poset and which is closed under linear sum and disjoint sum. The set $\mathcal{SP}$ is called the set of series-parallel posets.

\begin{theorem}\label{thm:finitesp}Let $P=(X,\leq)$ be a \emph{finite} poset. Then $P$ is $N$-free if and only if $P\in \mathcal{SP}$.
\end{theorem}

This theorem is false in the infinite case. Indeed, a poset whose comparability graph is $\mathcal{H}$ is $N$-free and yet it is connected and co-connected. On the other, hand we can replace the finiteness condition in Theorem \ref{thm:finitesp} by chain completeness or by the finiteness of the antichains to obtain a similar result, thus generalizing Theorem \ref{thm:finitesp}. A poset $P$ is \emph{chain complete}  if  every nonempty chain has a supremum and an infimum.

\begin{theorem}\label{thm:completesp}
Every chain complete $N$-free poset is a linear sum or a disjoint sum of chain complete $N$-free posets.
\end{theorem}

The following is a consequence of Corollary \ref{cor:1}.

\begin{theorem}\label{thm:finite-antichain}
Every $N$-free poset without infinite antichains is a linear sum or a disjoint sum of $N$-free posets without infinite antichains.
\end{theorem}

The proof of Theorems \ref{thm:completesp} and \ref{thm:finite-antichain} will be given in section \ref{section:proof-thms}.

A consequence of Theorem \ref{thm:finite-antichain} is a positive answer to a conjecture of Aharoni and Korman \cite{aharoni-korman} in the case of $N$-free posets with no infinite antichains. In 1992, Aharoni and Korman \cite{aharoni-korman} proposed the following conjecture:

\begin{conjecture}For an ordered set $P$ with no infinite antichains and any positive integer $k$, there are $k$ chains $C_1,\cdots, C_k$ and a partition of $P$ into antichains $(A_i : i \in I)$ such that each $A_i$ intersects $\min(|A_i|,k )$ chains $C_j$.
\end{conjecture}

Aharoni and Korman used K\"{o}nig duality theorem (\cite{konig1950} for the finite case and Aharoni \cite{aharoni84} for the infinite case) to prove  that Conjecture 1 is  true for posets of width $2$ (i.e., with no antichain of cardinality $3$). The instance $k=1$ of Conjecture 1 is known to be true for well founded and level finite posets. This follows from the Compactness Theorem of First Order Logic \cite{aharoni-korman}. Duffus and Goddard gave a constructive proof (see Theorem 2.1 \cite{duffs-goddard}). The instance $k=1$ of Conjecture  1 is also known to be true for other classes of posets with no infinite antichain, for instance for ordered sets of the form $C \times P$, where $C$ is a chain and $P$ is finite, and for ordered sets with no infinite
antichains and no infinite intervals. (see Duffus and Goddard \cite{duffs-goddard}).


\begin{theorem}\label{thm:aharoni-korman}
Every $N$-free poset with no infinite antichains possesses a chain and a partition into antichains so that each part intersects the chain.
\end{theorem}

The proof of Theorem \ref{thm:aharoni-korman} will be given in Section \ref{sec:proof-thm:aharoni-korman}.

\section{A generalisation of Theorem \ref{thm:1}}\label{general}

Let $G=(X,E)$ be a graph. A subset $A$ of $X$ is called a \emph{module} in $G$ if for every $v \not \in A$, either $v$ is adjacent to all vertices of $A$ or $v$ is not adjacent to any vertex of $A$. Clearly, the empty set, the singletons in $X$ and the whole set $X$ are modules in $G$; they are called
\emph{trivial}. A graph is called \emph{prime} if all its modules are trivial. With this definition, graphs on a set of size at most two are prime. Also, there are no prime graphs  on a three-element set.

The graph $P_4$ is prime. In fact, as it is well known, every prime graph contains an induced $P_4$ (Sumner \cite{sumner73} for finite graphs and Kelly \cite{kelly85} for infinite graphs). It follows that no $P_4$-free graph on at least 3 vertices is prime. 

\begin{lemma}\label{lem:1}Let $G=(X,E)$ be a connected $P_4$-free graph of any cardinality and $x\in X$.  Every connected component of $\inc(x)$ is a module in $X$.
\end{lemma}
\begin{proof}Let $C$ be a connected component of $\inc(x)$ and suppose for a contradiction that it is not a module in $G$. Then $C$ has at least two elements and there exists $y\not \in C$ such that $y$ is adjacent to some vertex of $C$ and not adjacent to another of vertex of $C$. Clearly we must have $y\in N(x)$. Now $y$ induces a partition of the vertices of $C$ into two subsets: the set $C_1$ of the vertices adjacent to $y$ and the set $C_2$ of those vertices not adjacent to $y$. Since $C$ is connected there are vertices $a\in C_1$ and $b\in C_2$ which are adjacent. But then $\{x,y,a,b\}$ is a $P_4$ in $G$ contradicting our assumption. The proof of the lemma is now complete.
\end{proof}

\begin{lemma}\label{lem:2}Let $G=(X,E)$ be a connected $P_4$-free graph of any cardinality and $x\in X$. Every connected component of $\inc(x)$ induces a partition of $N(x)$ into two subsets $N_1(x)$ and $N_2(x)$ such that every element of $N_1(x)$ is adjacent to every element of $N_2(x)$.
\end{lemma}
\begin{proof}Let $C$ be a connected component of $\inc(x)$. Set $N_1(x):= \{y\in N(x) : \mbox{ for all } c\in C, \{y,c\}\in E\}$ and $N_2(x):= \{y\in N(x) : \mbox{ for all } c \in C, \{y,c\}\not \in E\}$. Clearly, $\{N_1(x),N_2(x)\}$ is a partition of $N(x)$ (follows from Lemma \ref{lem:1}). Now suppose for a contradiction that there are vertices $y_1\in N_1(x)$ and $y_2\in N_2(x)$ such that $\{y_1,y_2\}\not \in E$. Then $\{c,y_1,x,y_2\}$ is a $P_4$ in $G$ where $c\in C$. A contradiction. The proof of the lemma is now complete.
\end{proof}

We now proceed to the proof of Theorem \ref{thm:2}.

\begin{proof}$\Rightarrow$ Suppose $N_x\neq \emptyset$ and set $N'_x:= X\setminus N_x$ and note that $x\in N'_x$. We prove that every vertex of $N_x$ is adjacent to every vertex of $N'_x$. All we have to do is to prove that every vertex of $N_x$ is adjacent to every vertex of $N(x)\setminus N_x$. So lets assume for a contradiction that there exist $y_1\in N_x$ and $y_2\in N(x)\setminus N_x$ such that $\{y_1,y_2\}\not \in E$. Since $y_2\not \in N_x$ there must be a connected component $C$ of $\inc(x)$ and $c\in C$ such that $\{y_2,c\}\not \in E$. But then $\{y_2,x,y_1,c\}$ is a $P_4$ in $G$ which is impossible.\\
$\Leftarrow$ Now suppose that ${\overline G}$ is not connected. There exists then a partition $\{X_1,X_2\}$ of $X$ such that every vertex of $X_1$ is adjacent to every vertex of $X_2$ and let $x\in X_1$. It follows then that $\inc(x)\in X_1$. Hence, there exists $y\in X_2$, and therefore in $N(x)$, which is adjacent to all vertices of $\inc(x)$, that is, $N_x\neq \emptyset$. The proof of the theorem is now complete.
\end{proof}

We now proceed to the proof of Corollary \ref{cor:1}.

\begin{proof}We prove that $N_x\neq \emptyset$ holds in the case where $G_{\restriction \inc(x)}$ has finitely many connected components. According to Lemma \ref{lem:1}, every connected component of $\inc(x)$ induces a partition of $N(x)$ into two subsets $N_1(x)$ and $N_2(x)$ such that every element of $N_1(x)$ is adjacent to every element of $N_2(x)$ and every vertex of $N_1(x)$ is adjacent to all vertices of $\inc(x)$. Among all connected components of $\inc(x)$ choose one, say $C$, that has the smallest (with respect to set inclusion) $N_1(x)$. We now prove that every vertex of $N_1(x)$ is adjacent to every vertex not in $N_1(x)$. It is enough to prove that every vertex in $N_1(x)$ is adjacent to every vertex in every connected component of $\inc(x)$. Suppose not and let $C'$ be a connected component of $\inc(x)$ and let $c'\in C'$ such that $\{y_1,c'\}\not \in E$ for some $y_1\in N_1(x)$. Let $y_2\in N(x)\setminus N_1(x)$. Then we must have $\{y_2,c'\}\in E$ because otherwise the partition induced by $C'$ on $N(x)$ generates a smaller $N_1(x)$ than that of $C$ contradicting our choice of the connected component $C$. Note that $\{y_2,y_1\}\in E$. But then $\{c',y_2,y_1,c\}$ is a $P_4$ for every $c\in C$ which is impossible. The proof of corollary is now complete.
\end{proof}

\section{A proof of Theorems \ref{thm:completesp} and \ref{thm:finite-antichain}}\label{section:proof-thms}

A subset $A$ of $X$ is called a {\it module} in a poset $P=(X,\leq)$ if for all $v\not\in A$ and for all $a,a^{\prime}\in A$
\begin{equation}
(v<a\Rightarrow v < a^{\prime})\;\mathrm{and}\;(a<v\Rightarrow a^{\prime} < v).
\end{equation}
The empty set, the singletons and the whole set $X$ are modules and are said to be {\it trivial}. The notion of a module for general relational structures was introduced by Fra\"{\i}ss\'{e} \cite{fraisse} who used the term "interval" rather than module. A poset is \emph{prime} if all its modules are trivial. We recall the following result (see \cite{kelly85}). 

\begin{theorem}\label{kelly} A poset $P$ is prime if and only if $\comp(P)$ is prime.
\end{theorem}

It follows from Proposition \ref{prop:1} and Theorem \ref{kelly} and the fact that no $P_4$-free  graph on at least three vertices is prime that no $N$-free poset on at least 3 vertices is prime. We should mention that every module of $P$ is a module of $\comp(P)$. The converse is false.

Let $P=(V,\leq)$ be a poset. For $x\in V$ we denote by $\inc(x)$ the set of all elements of $V$ that are incomparable to $x$ in $P$.

\begin{lemma}\label{lem:3}Let $P=(X,\leq)$ be a poset. Then $P$ is $N$-free if and only if for all $x\in X$ every connected component of $\inc(x)$ is a module in $P$.
\end{lemma}
\begin{proof} Assume that every connected component of $\inc(x)$ is a module and suppose for a contradiction that $(X,\leq)$ has an $N$. Then there are $a,b,c,d$ such that $(a,b,c,d)$ is an $N$. Hence, $\{c,d\}\subseteq inc(a)$ and $b>c$ but $b$ is incomparable to $d$ contradicting our assumption that the connected component of $inc(a)$ containing $\{c,d\}$ is a module in $(X,\leq)$.

Now let $x\in X$ and let $C$ be a connected component of $\inc(x)$ and suppose for a contradiction that it is not a module in $(X,\leq)$. Note that $C$ has at least two elements. There exists then $y\not \in \inc(x)$ that partitions $C$ into two nonempty subsets: the subset $C_1$ consisting of those elements of $C$ comparable to $y$ and the subset $C_2$ consisting of those elements of $C$ incomparable to $y$. Note that since $C$ is convex $y$ is either smaller than all elements of $C_1$ or is larger than all elements of $C_1$, say $y$ is lesser than all elements of $C_1$. We claim that $C_1$ and $C_2$ are connected components of $C$ contradicting the connectedness of $C$. Indeed, if not, then there exists $z\in C_1$ and $t\in C_2$ such that $t<z$. But then $(t,z,y,x)$ is an $N$ in $(X,\leq)$ contradicting our assumption.
\end{proof}

Let $P=(V,\leq)$ be a poset and $x\in V$. We define
\[C^{+}_x:=\{y : y > x \mbox{ and } y \mbox{ is comparable to all elements of } \inc(x)\},\]
and
\[C^{-}_x:=\{y : y < x \mbox{ and } y \mbox{ is comparable to all elements of } \inc(x)\}.\]

Set $C_x:=C^{-}_x\cup C^{+}_x$. 

\begin{lemma}\label{lem:4}
Let $P=(V,\leq)$ be a connected $N$-free poset. There exists $x\in X$ such that $C_x\neq \emptyset$ if and only if $P$ is a linear sum.
\end{lemma}
\begin{proof}Suppose $C_x\neq \emptyset$ and set
\[C'_x:=\{y : y \mbox{ is comparable to } x \mbox{ and } y \mbox{ is incomparable to some element of } \inc(x)\}.\]
Clearly $\{C_x,C'_x\}$ is a partition of the set of elements comparable to $x$. We prove that every element of $C^{-}_x$ is below every element of $C'_x$ and that every element of $C'_x$ is below every element of $C^{+}_x$. We notice at once that no element of $C'_{x}$ is above some element of $C^+_{x}$. Suppose for a contradiction that there exists $z>x$ and $z\in C'_x$ such that $z$ is incomparable to some element $y\in C^{+}_x$. Let $t\in \inc(x)$ be such that $t<y$. Then $(t,y,x,z)$ is an $N$ in $P$ contradicting our assumption that $P$ is $N$-free. This proves that every element of $C'_x$ is below every element of $C^{+}_x$. Similarly we prove that every element of $C'_x$ is above every element of $C^{+}_x$. Setting $P_1=C^{-}_x$, $P_2=C'_x\cup \inc(x)$, and $P_3=C^{+}_x$ we now have that $P$ is the linear sum of $P_1,P_2,P_3$ over the total order $1<2<3$.

Now suppose that $P$ is a linear sum, that is, $P=\displaystyle \sum_{i\in I}P_i$ and $I$ is totally ordered. Let $i\in I$ and $x\in P_{i}$. Then $\inc(x)\subseteq P_i$. Hence every element $y\in P_j$ for $j\neq i$ is necessarily in $C_x$ proving that $C_x\neq \varnothing$ as required.
\end{proof}


We now proceed to the proof of Theorem \ref{thm:completesp}.

\begin{proof}Let $P=(V,\leq)$ be a connected $N$-free poset and let $x\in X$. Now assume $P$ to be chain complete and let $C$ be a maximal chain containing $x$ and let $\mathfrak{s}$, respectively $\mathfrak{i}$, be its supremum, respectively its infimum. Then $\mathfrak{s}$ or $\mathfrak{i}$ is comparable to all elements of $\inc(x)$. If not, then $\mathfrak{s}$ is incomparable to some element $y_\mathfrak{s}$ of $\inc(x)$ and $\mathfrak{i}$ is incomparable to some element $y_\mathfrak{i}$ of $\inc(x)$. From the connectedness of $P$ and our assumption that $P$ is $N$-free we deduce that there exists $z$ comparable to $\mathfrak{s}$ and $y_\mathfrak{s}$. Then $z\not \in \inc(x)$ because otherwise $z$ is in the connected component $C(y_\mathfrak{s})$ of $\inc(x)$ containing $y_\mathfrak{s}$ and hence $y_\mathfrak{s}$ is comparable to $\mathfrak{s}$ since $C(y_\mathfrak{s})$ is a module (see Lemma \ref{lem:3}) contradicting our assumption that $\mathfrak{s}$ is incomparable to $y_{\mathfrak{s}}$. Clearly we must have $z<x$. Similarly there exists $z'$ comparable to $\mathfrak{i}$ and $y_\mathfrak{i}$ and $x<z'$. Assume that $z'$ is incomparable to $\mathfrak{s}$. Then $y_\mathfrak{i}<\mathfrak{s}$ because otherwise $(y_\mathfrak{i},z',x,\mathfrak{s})$ is an $N$ in $P$. In particular, $y_\mathfrak{s}$ and $y_\mathfrak{i}$ are incomparable (this is because $y_\mathfrak{s}$ is incomparable to $\mathfrak{s}$ and $y_\mathfrak{i}<\mathfrak{s}$ and hence by Lemma \ref{lem:3} $y_\mathfrak{s}$ and $y_\mathfrak{i}$ are in different connected components of $\inc(x)$). Also, $z'$ is incomparable to $y_s$ because otherwise $(y_\mathfrak{s},z',y_\mathfrak{i},\mathfrak{s})$ is a $N$ in $P$. But then $z<y_{\mathfrak{i}}$ because otherwise $(y_\mathfrak{i},z',z,y_\mathfrak{s})$ is an $N$ in $P$. Then $\mathfrak{i}<z$ because otherwise $(\mathfrak{i},x,z,y_{\mathfrak{i}})$ is an $N$ in $P$ which is impossible. But then $\mathfrak{i}<y_{\mathfrak{i}}$ since $z< y_{\mathfrak{i}}$ contradicting our assumption that $\mathfrak{i}$ is incomparable to $y_{\mathfrak{i}}$. Hence our assumption that $z'$ is incomparable to $y_\mathfrak{s}$ is false, that is, $z'<\mathfrak{s}$. Similarly we prove that $z$ is comparable to $\mathfrak{i}$, that is, $\mathfrak{i}<z$. Hence $y_\mathfrak{s}$ and $y_\mathfrak{i}$ are incomparable (this is because $y_\mathfrak{s}$ and $y_\mathfrak{i}$ are in different connected components of $\inc(x)$). But then $(y_\mathfrak{i},\mathfrak{s},\mathfrak{i},y_\mathfrak{s})$ is an $N$ in $P$ and this is impossible. Therefore $\mathfrak{s}$ or $\mathfrak{i}$ is comparable to all elements of $\inc(x)$. The conclusion of the theorem follows from Lemma \ref{lem:4}. The proof of the theorem is now complete.
\end{proof}

We now proceed to the proof of Theorem \ref{thm:finite-antichain}.

\begin{proof}Let $P=(V,\leq)$ be an $N$-free poset with no infinite antichains. Suppose $P$ is connected. Then $\comp(P)$ is a connected $P_4$-free graph. Let $x\in V$. Since $P$ has no infinite antichains we infer that $\inc(x)$ has finitely many connected components in $P$. These are the connected components of $\inc(x)$ in $\comp(P)$. Hence, $\comp(P)$ verifies the conditions of Corollary \ref{cor:1}. It follows then that $\incg(P)$ is not connected. The conclusion now follows from Lemma \ref{lem:folklore}.
\end{proof}

\section{A proof of Theorem \ref{thm:aharoni-korman}}\label{sec:proof-thm:aharoni-korman}

Let $P=(V,\leq)$ be a poset with no infinite antichains. Following \cite{abraham}, we order the set $\mathcal{A}(P)$ of all nonempty (finite) antichains with reverse of set inclusion. So $(\mathcal{A}(P),\supseteq)$ is a well founded poset: there is no infinite strictly decreasing sequence of antichains. Hence, a rank function $r : \mathcal{A}(P) \rightarrow \alpha$ is defined.
\[r(B) = \sup\{r(A) :   B\subset  A\in \mathcal{A}(P)\}.\]
The length of the antichain poset is denoted $ht(\mathcal{A}(P))$.
\[ht(\mathcal{A}(P))= \sup\{r(A): A\neq \varnothing \wedge A\in \mathcal{A}(P)\} = \sup\{r(\{v\}) : v\in V\}.\]
as singletons are antichains.

So for example for $n\in \NN$, $ht(\mathcal{A}(P))=n$ if and only if $n$ is the maximal size of an antichain in $P$. The following is Lemma 1.10 of \cite{abraham}.

\begin{lemma}\label{lem:rank-inc}Let $P=(V,\leq)$ be a poset with no infinite antichains and $v\in V$. If $P$ is not an antichain, then $ht(inc(v))<ht(\mathcal{A}(P))$.
\end{lemma}

The decomposition of the incomparability graph of a poset into connected components is expressed in the following lemma which belongs
to the folklore of the theory of ordered sets.

\begin{lemma}\label{lem:folklore} If $P:= (V, \leq)$ is a poset, the order on $P$ induces a total order on the set $Connect(P)$
of connected components of $\incg(P)$ and $P$ is the lexicographical sum of these components indexed by the chain $Connect(P)$. In
particular, if $\preceq$ is a total order extending the order $\leq$ of $P$, each connected component $A$ of $\incg(P)$ is an interval of
the chain $(V, \preceq)$.
\end{lemma}

The proof of the following lemma is easy and is left to the reader.

\begin{lemma}\label{lem:incomp-not-con}Let $P=\sum_{i\in I}P_i$ be a linear sum. Let $\mathcal{A}_i$ be a partition of $P_i$ into antichains and let $\mathcal{C}_i$ be a chain of $P_i$ so that $\mathcal{C}_i$ meets each part of $\mathcal{A}_i$. Then $\bigcup_{i\in I}\mathcal{A}_i$ is a partition of $P$ into antichains and $\sum_{i\in I}\mathcal{C}_i$ is a chain of $P$ that meets every part of $\bigcup_{i\in I}\mathcal{A}_i$.
\end{lemma}

It follows easily from Lemmas \ref{lem:folklore} and \ref{lem:incomp-not-con} that if the Aharoni--Korman Conjecture is true for posets with no infinite antichains and whose incomparability graphs are connected, then it is true for all posets with no infinite antichains.

We now proceed to the proof of Theorem \ref{thm:aharoni-korman}.
\begin{proof}
Let $P$ be an $N$-free poset with no infinite antichains. The proof is by induction on the antichain rank of $P$. It follows from Theorem \ref{thm:finite-antichain} that $P$ is either a linear sum or a disjoint sum of $N$-free posets with no infinite antichains. It follows from Lemmas \ref{lem:folklore} and \ref{lem:incomp-not-con} that we can assume $P$ to be a disjoint sum of $N$-free posets with no infinite antichains. Since $P$ has no infinite antichains we infer that $P$ has finitely many connected components $P_1,\ldots,P_n$. Now each $P_i$ is a subset of $\inc(x)$ for every $x\in P_j$ and $j\neq i$. Hence, the antichain rank of every connected component is smaller than that of $P$ and hence we can apply the induction hypothesis. For $1\leq i\leq n$, let $\mathcal{A}_i$ be a partition of $P_i$ into antichains and let $\mathcal{C}_i$ be a chain of $P_i$ that intersects every member of the partition $\mathcal{A}_i$. Among all chains $\mathcal{C}_i$ choose one of maximum cardinality, say $\mathcal{C}_1$ has maximum cardinality $\lambda_1$, and we may assume without loss of generality that if $\lambda_i=|\mathcal{C}_i|$, then $\lambda_1\geq \lambda_2 \cdots \geq \lambda_n$. We view each $\lambda_i$ as an enumeration of $\mathcal{C}_i$ for $1\leq i\leq n$. For $2\leq i\leq n$, we let $f_i$ be a one-to-one map of $\lambda_i$ into $\lambda_1$. Hence, each $f_i$ maps an element of $C_i$ into an element of $C_1$. For $2\leq i\leq n$ and for all $c\in \mathcal{C}_1$, we let $A_{f_i^{-1}(c)}$ be the unique antichain in $\mathcal{A}_i$ containing $f_i^{-1}(c)$. We let $A_{f_1^{-1}(c)}$ be the unique antichain in $\mathcal{A}_1$ containing $c$. For each $c\in \mathcal{C}_1$ set $A'_c:=\bigcup_{1\leq i\leq n} A_{f_i^{-1}(c)}$. Set $\mathcal{A}':=\bigcup_{c\in C_1}A'_c$. Then clearly $\mathcal{A}'$ is a partition of $P$ into antichains and $C_1$ meets each part of $\mathcal{A}'$.
\end{proof}

\end{document}